\documentclass[12pt]{amsproc}
\newtheorem{theorem}{\sc Theorem}[section]
\newtheorem{lemma}[theorem]{\sc Lemma}

\newtheorem{hypothesis}[theorem]{\sc Hypothesis}

\author{Pavel Shumyatsky}
\address{Department of Mathematics, University of Brasilia,
Brasilia-DF, 70910-900 Brazil}
\email{pavel@unb.br}

\title[Coverings by nilpotent subgroups]{On profinite groups with commutators covered by nilpotent subgroups}
\thanks{This work was  supported by CAPES and CNPq.}

\begin{document}

\begin{abstract} The following results about a profinite group $G$ are obtained. The commutator subgroup $G'$ is finite if and only if $G$ is covered by countably many abelian subgroups. The group $G$ is finite-by-nilpotent if and only if $G$ is covered by countably many nilpotent subgroups. The main result is that the commutator subgroup $G'$ is finite-by-nilpotent if and only if the set of all commutators in $G$ is covered by countably many nilpotent subgroups.
\end{abstract}
\maketitle

\section{Introduction}

Let $G$ be a profinite group. If $G$ is covered by countably many closed subgroups, then by Baire Category Theorem at least one of the subgroups is open. This simple observation suggests that if $G$ is covered by countably many closed subgroups with certain specific properties, then the structure of $G$ is similar to that of the covering subgroups. For example, if $G$ is covered by countably many periodic subgroups, then $G$ is locally finite. We recall that the group $G$ is periodic (or torsion) if each element in $G$ has finite order. The group is locally finite if each of its finitely generated subgroups is finite. Following his solution of the restricted Burnside problem \cite{ze1,ze2} and using Wilson's reduction theorem \cite{Wil}, Zelmanov proved that periodic compact groups are locally  finite \cite{Zel}. Another example is that if $G$ is covered by countably many subgroups of finite rank, then $G$ has finite rank. The profinite group $G$ is said to have finite rank at most $r$ if each closed subgroup of $G$ can be topologically generated by at most $r$ elements. It was shown in the recent paper \cite{as2} that a profinite group is covered by countably many procyclic subgroups if and only if it is finite-by-procyclic.

In the present article we will establish more results of this nature.

\begin{theorem}\label{a} For a profinite group $G$, the following conditions are equivalent.
\begin{enumerate}
\item The group $G$ is covered by countably many abelian subgroups;
\item The group $G$ has finite commutator subgroup;
\item The group $G$ is central-by-finite.
\end{enumerate}
\end{theorem}

\begin{theorem}\label{b} For a profinite group $G$, the following conditions are equivalent.
\begin{enumerate}
\item The group $G$ is covered by countably many nilpotent subgroups;
\item The group $G$ is finite-by-nilpotent;
\item There exists a positive integer $m$ such that $Z_m(G)$ is open.
\end{enumerate}
\end{theorem}

Here $Z_m(G)$ denotes the $m$th term of the upper central series of $G$. Both above theorems are in parallel with well-known results on abstract groups covered by finitely many abelian, or nilpotent, subgroups. R. Baer \cite[4.16]{rob1} characterized central-by-finite groups as those groups having a
finite covering by abelian subgroups. In 1992 M. J. Tomkinson showed that a group $G$ has a finite covering by nilpotent subgroups if and only if $Z_m(G)$ has finite index in $G$ for some positive $m$ \cite{tomki}. Some related issues have been also addressed in \cite{kappe}. In view of Hall's theorem \cite{hhh} it follows that a group $G$ has a finite covering by nilpotent subgroups if and only if $G$ is finite-by-nilpotent. It is interesting to observe that our Theorems \ref{a} and \ref{b} show that a profinite group $G$ admits a countable covering by abelian (respectivly, nilpotent) subgroups if and only if $G$ admits a finite covering by subgroups with the respective property. The proofs of the theorems do not use the results on finite coverings of abstract groups. 

If $x,y\in G$, then $[x,y]=x^{-1}y^{-1}xy$ is the commutator of $x$ and $y$.
The closed subgroup of $G$ generated by all commutators is the commutator subgroup $G'$ of $G$. In general, elements of $G'$ need not be commutators (see for instance \cite{km} and references therein). On the other hand, Nikolov and Segal showed that for any positive integer $m$ there exists an integer $f(m)$ such that if $G$ is $m$-generator, then every element in $G'$ is a product of at most $f(m)$ commutators \cite{nisega}. Several recent results  indicate that if the set of all commutators is covered by finitely, or countably, many subgroups with certain specific properties, then the structure of $G'$ is somehow similar to that of the covering subgroups.

It was shown in \cite{AS1} that if $G$ is a profinite group that has finitely many periodic subgroups (respectively, subgroups of finite rank) whose union contains all commutators, then $G'$ is locally finite (respectively, $G'$ is of finite rank). In \cite{AS0} similar results were obtained for the case where commutators are covered by countably many subgroups: if $G$ is a profinite group that has countably many periodic subgroups (respectively, subgroups of finite rank) whose union contains all commutators, then $G'$ is locally finite (respectively, $G'$ is of finite rank). In \cite{DMS} the corresponding results were obtained for profinite groups in which commutators of higher order are covered by countably many periodic subgroups, or subgroups of finite rank.

Profinite groups in which commutators are covered by procyclic subgroups were studied in \cite{FMS}. It was shown that if $G$ is a profinite group that has finitely many, say $m$, procyclic subgroups whose union contains all commutators, then $G'$ is finite-by-procyclic. In fact, $G'$ has a finite characteristic subgroup $M$ of $m$-bounded order such that $G'/M$ is procyclic. Moreover, if $G$ is a pro-$p$ group that has $m$ procyclic subgroups whose union contains all commutators, then $G'$ is either finite of $m$-bounded order or procyclic. Earlier, Fern\'andez-Alcober and Shumyatsky proved that if $G$ is an abstract group in which the set of all commutators is covered by finitely many cyclic subgroups, then the commutator subgroup $G'$ is either finite or cyclic \cite{FerShu}. Profinite groups in which the commutators are covered by countably many procyclic subgroups were dealt with in \cite{as2}. Such groups were characterized precisely as groups whose commutator subgroup is finite-by-procyclic.

In the present article we study profinite groups in which the commutators are covered by countably many nilpotent subgroups. Our main result is the following theorem.

\begin{theorem}\label{mainth} Let $G$ be a profinite group. The following conditions are equivalent.
\begin{enumerate}
\item The set of all commutators in $G$ is covered by countably many nilpotent subgroups;
\item The commutator subgroup $G'$ is finite-by-nilpotent;
\item There exists a positive integer $m$ such that $Z_m(G')$ is open in $G'$.
\end{enumerate}
\end{theorem}

Remark that no analogue of the above theorem for abstract groups is known. We conjecture that if $G$ is an abstract group in which the commutators are covered by finitely many nilpotent subgroups, then $G'$ is finite-by-nilpotent.

Another natural conjecture (related to Theorem \ref{a}) is that if $G$ is a profinite group in which the commutators are covered by countably many abelian subgroups, then the second commutator subgroup $G''$ is finite. It seems however that the techniques employed in the present article are insufficient for dealing with the above conjectures.

\section{On groups covered by nilpotent subgroups}
The proof of Theorem \ref{a} is quite straightforward. The proof of Theorem \ref{b} will be somewhat more complicated. Throughout the article, whenever $G$ is a profinite group, we denote by $\langle X\rangle$ the subgroup of $G$ topologically generated by the subset $X$.

\begin{proof}[Proof of Theorem \ref{a}] It is immediate to see that a profinite group $G$ has finite commutator subgroup if and only if it is central-by-finite. Indeed, if $G$ is central-by-finite, then by Schur's theorem \cite[Theorem 4.12]{rob1} the commutator subgroup $G'$ is finite. On the other hand, if $G'$ is finite, $G$ possesses an open normal subgroup $N$ such that $N\cap G'=1$. It is clear that $N\leq Z(G)$ and so $G$ is central-by-finite.

Further, suppose that $Z(G)$ is open in $G$. Choose representatives $g_1,\dots,g_n$ of the cosets of $Z(G)$ in $G$. We observe that $G$ is a union of $n$ abelian subgroups $\langle g_i,Z(G)\rangle$. Thus, the proof of the theorem will be complete once we show that if $G$ is a union of countably many abelian subgroups, then $Z(G)$ is open in $G$.

Write $G=\cup G_i$, where $i=1,2,\dots$ and $G_i$ are closed abelian subgroups. By Baire Category Theorem \cite[p. 200]{kell} at least one of the subgroups $G_i$ is open. So we assume that $G_1$ is open. Choose $g_1,\dots,g_n\in G$ such that $G=\langle G_1,g_1,\dots,g_n\rangle$. For some $i\leq n$ we look at the coset $B=g_iG_1$. It is clear that $B$ is covered by the subsets $B\cap G_j$. The subsets $B\cap G_j$ are closed and so by Baire Category Theorem at least one of them has non-empty interior. Therefore there exist a positive integer $k\geq2$, an element $b\in B$, and an open subgroup $D_i$ such that $bD_i\subseteq G_k$. Since $G_k$ is abelian, it follows that so is the subgroup $\langle b,D_i\rangle$. Taking into account that $g_i\in bG_1$, we deduce that $g_i$ centralizes $D_i\cap G_1$. Such a subgroup $D_i$ can be chosen for each $i\leq n$. We see that $G_1\cap\bigcap_{1\leq i\leq n}D_i$ is an open central subgroup of $G$. It follows that $Z(G)$ is open in $G$, as required.
\end{proof}
 
As usual, if $X$ and $Y$ are subsets of a group $G$, we denote by $[X,Y]$ the subgroup generated by all commutators $[x,y]$, where $x\in X$ and $y\in Y$. For $k=1,2,\dots$ we write $X=[X,{}_0Y]$ and $[X,{}_kY]=[[X,{}_{k-1}Y],Y]$.
\begin{lemma}\label{3b} Let $G$ be a finite group acting on an abelian group $M$. Suppose that there exists an integer $k$ with the property that for every element $a\in G$ we have $[M,{}_ka]=1$. Then there exists a number $t$, depending only on $k$ and $|G|$, such that $[M,{}_tG]=1$.
\end{lemma}
\begin{proof} This is straighforward using the result of Crosby and Traustason \cite[Theorem 1]{crtr}.
\end{proof}

\begin{lemma}\label{4b} Let $G$ be a profinite group covered by countably many nilpotent subgroups. Assume that $N$ is a closed normal abelian subgroup of $G$ such that $C_G(N)$ is open. Then $N$ contains a subgroup $M$ such that \begin{enumerate}
\item $M$ is normal in $G$;
\item $M$ is open in $N$;
\item There exists a number $t$ such that $[M,{}_tG]=1$.
\end{enumerate}
\end{lemma}
\begin{proof} Write $G=\cup G_i$, where $i=1,2,\dots$ and $G_i$ are closed nilpotent subgroups. Let $x\in G$ and $X=xC_G(N)$. Obviously, $X$ is closed and therefore compact. It is clear that $X$ is covered by the (closed) subsets $X\cap G_i$. By Baire category theorem,  at least one of these subsets contains a non-empty interior. Hence, there exist an open normal subgroup $T$ in $G$, an element $a\in X$, and an integer $j$  such that $X\cap aT$ is contained in $G_j$. Let $R=T\cap C_G(N)$. Since $aR$ is contained  in $G_j$, it follows that $\langle a,R\rangle\leq G_j$. Since $G_j$ is nilpotent, there exists a positive integer $k$ such that $[R,{}_ka]=1$. Thus, we have shown that for each coset $xC_G(N)$ there exist $a\in xC_G(N)$, a positive integer $k_a$, and an open normal subgroup $R_a$ such that $[R_a,{}_{k_a}a]=1$. We fix such representatives $a_1,a_2,\dots,a_s$, one in each coset of $C_G(N)$. Set $M=N\cap\bigcap_iR_{a_i}$. Being the intersection of $N$ with finitely many open normal subgroups, $M$ is normal in $G$ and open in $N$. Let $k_0$ be the maximum of the numbers $k_{a_i}$. The group $\bar{G}=G/C_G(N)$ naturally acts on $M$. Moreover, the action of $\bar{G}$ on $M$ verifies the condition $[M,{}_{k_0}\bar{a}]=1$ for every $\bar{a}\in\bar{G}$. By Lemma \ref{3b} there exists a number $t$ such that $[M,{}_t\bar{G}]=1$. Of course, this implies that $[M,{}_tG]=1$, as required.
\end{proof}

\begin{proof}[Proof of Theorem \ref{b}] In view of Hall's theorem \cite{hhh} it follows that $G$ is finite-by-nilpotent if and only if there exists a positive integer $m$ such that $Z_m(G)$ is open. Further, suppose that for some positive integer $m$ the subgroup $Z_m(G)$ is open in $G$. Choose representatives $g_1,\dots,g_n$ of the cosets of $Z_m(G)$ in $G$. We observe that $G$ is a union of $n$ nilpotent subgroups $\langle g_i,Z_m(G)\rangle$. Thus, the proof of the theorem will be complete once we show that if $G$ is a union of countably many nilpotent subgroups, then for some positive integer $m$ the subgroup $Z_m(G)$ is open in $G$.

So we assume that $G=\cup G_i$, where $i=1,2,\dots$ and $G_i$ are closed nilpotent subgroups. By Baire Category Theorem at least one of the subgroups $G_i$ is open. Therefore $G$ is virtually nilpotent. Let $c$ be the minimal nonnegative integer such that $G$ contains an open, nilpotent of class $c$, normal subgroup $N$. If $c=0$, then $G$ is finite and there is nothing to prove. If $c=1$, then $G$ is virtually abelian and the result is straightforward from Lemma \ref{4b}. Thus, suppose that $c\geq2$ and use induction on $c$. Put $Z=Z(N)$. It is clear that $C_G(Z)$ is open. Therefore, by Lemma \ref{4b}, there exists a number $t$ and a subgroup $M\leq Z$, which is normal in $G$ and open in $Z$, such that $[M,{}_tG]=1$. Since $Z/M$ is finite, the group $G$ has an open normal subgroup $K$ such that $K\cap Z\leq M$. Let $L=K\cap N$ and $F=M\cap L$. The nilpotency class of $L/F$ is at most $c-1$ and $[F,{}_tG]=1$. By induction, there exists an open normal subgroup $J$ such that $[J,{}_rG]\leq F$ for some positive integer $r$. Therefore $[J,{}_{r+t}G]=1$ and $J\leq Z_{r+t}(G)$. The proof is complete. 
\end{proof}

\section{On groups in which commutators are covered by nilpotent subgroups}
Our goal in the present section is to prove Theorem \ref{mainth}. Thus, we will work under the following hypothesis.
\begin{hypothesis}\label{aaa} Let $G$ be a profinite group in which the set of all commutators is contained in the union $\cup G_i$, where $i=1,2,\dots$ and $G_i$ are closed nilpotent subgroups.
\end{hypothesis}

\begin{lemma}\label{1a} Assume Hypothesis \ref{aaa}. Then $G$ has an open normal subgroup $H$ such that $H'$ is nilpotent.
\end{lemma}

\begin{proof}
For each positive integer $i$ set \[S_i=\{(x,y)\in G\times G \mid [x,y]\in G_i \}.\]
Note that the sets $S_i$ are closed in $G\times G$ and cover the whole group $G\times G$. By Baire category theorem at least one of these sets contains a non-empty interior. Hence, there exist an open normal subgroup $H$ in $G$,   elements $a,b \in G$, and an integer $j$  such that $[ah_1,bh_2]\in G_j$ for any choice of $h_1,h_2\in H$.

Let $K$ be the closed subgroup of $G$ generated by all commutators of the form $[ah_1,bh_2]$, where $h_1,h_2\in H$. Note that $K\le G_j$ and that $H$ normalizes $K$. Since $G_j$ is nilpotent, so is $K$. Let $D=K\cap H$. Then $D$ is a normal nilpotent subgroup of $H$ and the normalizer of $D$ in $G$ has finite index. Therefore there are only finitely many conjugates of $D$ in $G$. Let $D =D_1,D_2,\ldots, D_r$ be all these conjugates. All of them are normal in $H$ and so their product $D_1D_2\cdots D_r$ is nilpotent. Of course, the product $D_1D_2\cdots D_r$ is a closed normal subgroup of $G$. Thus, we can examine the quotient $G/D_1D_2\cdots D_r$.

Suppose that $D=1$. This implies that $[ah_1,bh_2]=[a,b]$ for any $h_1,h_2\in H$. Obviously this happens if and only if $[H,a]=[H,b]=H'=1$. Therefore $H'\leq D_1D_2\cdots D_r$ and the lemma follows.
\end{proof}
\begin{lemma}\label{2a} Assume Hypothesis \ref{aaa}. Then for every element $a\in G$ there exists an open normal subgroup $H_a$ such that $[H_a,a]$ is nilpotent.
\end{lemma}
\begin{proof} Fixed an element $a\in G$, for each positive integer $i$ let \[S_i=\{x\in G\mid [x,a]\in G_i\}.\]
Note that the sets $S_i$ are closed in $G$ and cover the whole group $G$. By Baire category theorem at least one of these sets contains a non-empty interior. Hence, there exist an open normal subgroup $H$ in $G$, an element $b\in G$, and an integer $j$  such that $[hb,a]\in G_j$ for any $h\in H$. We have $[hb,a]=[h,a]^b[b,a]$. Since $[b,a]\in G_j$, we conclude that $[h,a]^b\in G_j$ for any  $h\in H$. Therefore $[H,a]\leq G_j^{b^{-1}}$, which is nilpotent.
\end{proof}

\begin{lemma}\label{3a} Let $G$ be a finite group acting on an abelian group $M$. Suppose that there exists an integer $k$ such that for every commutator $a\in G$ we have $[M,{}_ka]=1$. Then there exists a number $t$, depending only on $k$ and $|G'|$, such that $[M,{}_tG']=1$.
\end{lemma}
\begin{proof} Choose $x\in M$. The subgroup $\langle x^G\rangle$ is $G$-invariant. It is sufficient to show that there exists a number $t$, depending only on $k$ and $|G'|$ (and not on the choice of $x\in M$), such that $[\langle x^G\rangle,{}_tG']=1$. Therefore we can assume that $M=\langle x^G\rangle$. Now $M$ is finitely generated and hence residually finite. It is sufficient to show that there exists a number $t$, depending only on $k$ and $|G'|$, such that $[Q,{}_tG']=1$ whenever $Q$ is a finite $G$-invariant quotient of $M$. Therefore we can assume that $M$ is finite. Further, we note that $G$ naturally acts on each Sylow subgroup of $M$. It is sufficient to prove that the action of $G$ on each Sylow subgroup of $M$ satisfies the conclusion of the lemma. Hence, without loss of generality we assume that $M$ is a finite $p$-group for some prime $p$.

 The lemma will be proved by induction on $|G'|$. Since there is nothing to prove when $G'=1$, we assume that $G'\neq1$. Suppose that $G''\neq G'$. By induction, there exists a $(k,|G''|)$-bounded number $s$ such that $[M,{}_sG'']=1$. We allow the case where $G''=1$ and $s=1$, and we assume that $s$ is the minimal number with the property that $[M,{}_sG'']=1$. Set $M_i=[M,{}_iG'']$ for $i=0,1,\dots$. Of course, we have $M_0=M$ and $M_{i+1}=[M_i,G'']$ for $i=0,1,\dots$. Let $\bar{G}=G/C_G(M_{s-1})$ and we remark that $\bar{G}$ is metabelian since $G''\leq C_G(M_{s-1})$. The group $\bar{G}$ naturally acts on $M_{s-1}$. Let $a_1,a_2,\dots,a_r$ be the commutators in $\bar{G}$. Of course their number $r$ does not exceed $|G'|$. For $i=1,2,\dots,r$ we let $H_i=\langle a_i,M_{s-1}\rangle$. Since for every commutator $a\in G$ we have $[M,{}_ka]=1$, it follows that each subgroup $H_i$ is nilpotent of class at most $k$. Because $\bar{G}$ is metabelian, the subgroups $H_i$ normalize each other and therefore the subgroup $\langle a_1,a_2,\dots,a_r,M_{s-1}\rangle$ is nilpotent of class at most $kr$. In particular, $[M_{s-1},{}_{kr}\bar{G}']=1$. Of course, it follows that $[M_{s-1},{}_{kr}G']=1$.

Now we look at the natural action of $G$ on $M/M_{s-1}$. If $s\geq2$, we repeat the above argument and obtain that $[M_{s-2}/M_{s-1},{}_{kr}G']=1$. In other words, $[M_{s-2},{}_{kr}G']\leq M_{s-1}$. Hence, $[M_{s-2},{}_{2kr}G']=1$. Next, we look at the natural action of $G$ on $M/M_{s-2}$ and so on. Eventually we obtain that $[M,{}_{skr}G']=1$. Thus, in the case where $G''\neq G'$ the result follows. Now assume that $G''=G'$.

Let $K$ be the subgroup of $G$ generated by all commutators whose orders are not divisible by $p$. Then all commutators in $G/K$ have $p$-power order and it follows that $G'/K$ is a $p$-group (see for example Lemma 3.2 in \cite{proc99}). Since $G''=G'$, we conclude that $G'=K$. Choose a commutator $a\in G$ of order not divisible by $p$. The well-known property of coprime automorphisms (see \cite[Theorem 5.3.2]{gore}) shows that $[M,{}_ka]=1$ if and only if $[M,a]=1$. Taking into account that $G'$ is generated by commutators of order not divisible by $p$ we deduce that $[M,G']=1$. The proof is now complete.
\end{proof}
\begin{lemma}\label{4a} Assume Hypothesis \ref{aaa} and suppose that $N$ is a closed normal abelian subgroup of $G$ such that $C_G(N)\cap G'$ is open in $G'$. Then $N$ contains a subgroup $M$ such that \begin{enumerate}
\item $M$ is normal in $G$;
\item $M$ is open in $N$;
\item There exists a number $t$ such that $[M,{}_tG']=1$.
\end{enumerate}
\end{lemma}
\begin{proof} Let $x\in G$ and let $X$ be the set of all commutators contained in the coset $xC_G(N)$. Suppose that $X$ is non-empty. Obviously, the set $X$ is closed and therefore compact. It is clear that $X$ is covered by the (closed) subsets $X\cap G_i$. By Baire category theorem,  at least one of these subsets contains a non-empty interior. Hence, there exist an open normal subgroup $T$ in $G$, an element $a\in X$, and an integer $j$  such that all commutators contained in $X\cap aT$ belong to $G_j$. Let $R=T\cap C_G(N)$. Since all commutators contained in $aR$ belong to $G_j$, it follows that $\langle a^R\rangle\leq G_j$. Since $G_j$ is nilpotent, there exists a positive integer $k$ such that $[R,{}_ka]=1$. Thus, we have shown that whenever the coset $xC_G(N)$ contains commutators, there exists a commutator $a\in xC_G(N)$, a positive integer $k_a$, and an open normal subgroup $R_a$ such that $[R_a,{}_{k_a}a]=1$. We fix such commutators $a_1,a_2,\dots,a_s$, one in each coset of $C_G(N)$ containing commutators. Set $M=N\cap\bigcap_iR_{a_i}$. Being the intersection of $N$ with finitely many open normal subgroups, $M$ is normal in $G$ and open in $N$. Let $k_0$ be the maximum of the numbers $k_{a_i}$. The group $\bar{G}=G/C_G(N)$ naturally acts on $M$. Further, the derived group of $\bar{G}$ is finite and the action of $\bar{G}$ on $M$ verifies the condition $[M,{}_{k_0}\bar{a}]=1$ whenever $\bar{a}$ is a commutator in $\bar{G}$. By a profinite version of Lemma \ref{3a} there exists a number $t$ such that $[M,{}_t\bar{G}']=1$. Of course, this implies that $[M,{}_tG']=1$, as required.
\end{proof}

Now we are ready to prove Theorem \ref{mainth}.
\begin{proof}[Proof of Theorem \ref{mainth}] We know from the proof of Theorem \ref{b} that $G'$ is finite-by-nilpotent if and only if there exists a positive integer $m$ such that $Z_m(G')$ is open in $G'$ and that $G'$ is a union of finitely many nilpotent subgroups whenever $Z_m(G')$ is open in $G'$. Therefore we only have to show that under Hypothesis \ref{aaa} there exists a positive integer $m$ such that $Z_m(G')$ is open in $G'$.

Thus, assume Hypothesis \ref{aaa}. By Lemma \ref{1a} $G$ possesses an open normal subgroup $H$ such that $H'$ is nilpotent. Let $a_1,a_2,\dots,a_s$ be elements of $G$ such that $G=\langle H,a_1,a_2,\dots,a_s\rangle$. By Lemma \ref{2a} for every $i=1,\dots,s$ there exists an open normal subgroup $H_i$ such that $[H_i,a_i]$ is nilpotent. Let $U$ be the product of $H'$ and all subgroups $[H_i,a_i]$ for $i=1,\dots,s$. Then $U$ is nilpotent. Let $V$ be the intersection of $H$ and the subgroups $H_i$ for $i=1,\dots,s$. Thus, $V$ is an open normal subgroup and it is easy to see that $V/U$ is centralized by $H$ and the generators $a_1,a_2,\dots,a_s$. We conclude that $V/U$ is central in $G/U$. By Schur's theorem $G'/U$ is finite.

Thus, we have shown that $G'$ is virtually nilpotent. Let $c$ be the minimal nonnegative integer such that $G'$ contains an open, nilpotent of class $c$, subgroup $N$ which is normal in $G$. If $c=0$, then $G'$ is finite and there is nothing to prove. If $c=1$, then $G'$ is virtually abelian and the result is straightforward from Lemma \ref{4a}. Thus, assume that $c\geq2$ and use induction on $c$. Put $Z=Z(N)$. It is clear that $C_G(Z)\cap G'$ is open in $G'$. Therefore, by Lemma \ref{4a}, there exists a number $t$ and a subgroup $M\leq Z$, which is normal in $G$ and open in $Z$, such that $[M,{}_tG']=1$. Since $Z/M$ is finite, the group $G$ has an open normal subgroup $K$ such that $K\cap Z\leq M$. Let $L=K\cap N$ and $F=M\cap L$. The nilpotency class of $L/F$ is at most $c-1$ and $[F,{}_tG']=1$. By induction, there exists an open normal subgroup $J$ in $G'$ such that $[J,{}_rG']\leq F$ for some positive integer $r$. Therefore $[J,{}_{r+t}G']=1$ and $J\leq Z_{r+t}(G')$. The proof is now complete.
\end{proof}

\end{document}